\documentclass[12pt,leqno,twoside]{amsart}

\usepackage[latin1]{inputenc}
\usepackage[T1]{fontenc}
\usepackage[colorlinks=true, pdfstartview=FitV, linkcolor=blue, citecolor=blue, urlcolor=blue]{hyperref}
\usepackage{amstext,amsmath,amscd, bezier,indentfirst,amsthm,amsgen,enumerate, geometry}
\usepackage[all,knot,arc,import,poly]{xy}
\usepackage{amsfonts,color, soul}  
\usepackage{amssymb}
\usepackage{latexsym}
\usepackage{epsfig}
\usepackage{graphicx}
\usepackage{srcltx}
\usepackage{enumitem}

\newtheorem{theorem}{Theorem}[section]
\newtheorem{corollary}[theorem]{Corollary}
\newtheorem{lemma}[theorem]{Lemma}
\newtheorem{proposition}[theorem]{Proposition}
\theoremstyle{definition}
\newtheorem{definition}[theorem]{Definition}
\theoremstyle{remark}
\newtheorem{remark}[theorem]{\sc Remark}
\newtheorem{example}[theorem]{\sc Example}

\renewcommand{\Box}{\square}    

\renewcommand{\i}{{\rm{int}}}

\newcommand{\Sing}{{\mathrm{Sing\hspace{2pt}}}}

\newcommand{\Disc}{{\mathrm{Disc\hspace{2pt}}}}
\newcommand{\ord}{{\mathrm{ord}}}

\newcommand{\im}{{\mathrm{Im\hspace{2pt}}}}

\newcommand{\ity}{{\infty}}

\newcommand{\NT}{{\mathrm{NT}}}

\newcommand{\e}{\varepsilon}
\newcommand{\m}{\setminus}
\newcommand{\s}{\subset}

\newcommand{\fin}{\hspace*{\fill}$\Box$\vspace*{2mm}}

\newcommand{\cC}{{\mathcal C}}

\newcommand{\bR}{{\mathbb R}}
\newcommand{\bC}{{\mathbb C}}

\newcommand{\bN}{{\mathbb N}}

\newcommand{\bK}{{\mathbb K}}
\newcommand{\bH}{{\mathbb H}}
\newcommand{\bS}{{\mathbb S}}

\newcommand{\bW}{{\mathbb W}}

\begin{document}

		 


\title[Fibrations of highly singular map germs]{Fibrations of highly singular map germs}

\author{\sc Raimundo N. Ara\'ujo dos Santos}
\address{ICMC,
Universidade de S\~ao Paulo,  Av. Trabalhador S\~ao-Carlense, 400 -
CP Box 668, 13560-970 S\~ao Carlos, S\~ao Paulo,  Brazil}
\email{rnonato@icmc.usp.br, raimundo.icmc@gmail.com}

\author{\sc Maico F. Ribeiro}
\address{ICMC,
Universidade de S\~ao Paulo,  Av. Trabalhador S\~ao-Carlense, 400 -
CP Box 668, 13560-970 S\~ao Carlos, S\~ao Paulo,  Brazil}
\email{maico.ribeiro@ufes.br}

\author{Mihai Tib\u ar}
\address{Univ. Lille, CNRS, UMR 8524 -- Laboratoire Paul Painlev\'e, F-59000 Lille,
France}  
\email{tibar@math.univ-lille1.fr}



\thanks{The authors acknowledge the support of USP-COFECUB Uc Ma 163-17.  M. Tibar acknowledges
the support of  Labex CEMPI (ANR-11-LABX-0007-01), and  R.N. Ara\'ujo dos Santos acknowledges the support of Fapesp grant 2017/20455-3 and CNPq research-grant 313780/2017-0.}

\begin{abstract}
 We define  \emph{local fibration structures} for  real map germs with strictly positive dimensional discriminant:
 a local fibration structure over the complement of the discriminant, and a complete local fibration structure which includes the stratified discriminant into the picture. We provide new classes of map germs endowed with
such local fibration structures.

\end{abstract}

\maketitle
\section{Introduction}

 The existence of a local fibration structure defined by a map germ $G:(\bK^{m},0) \rightarrow (\bK^{p}, 0)$, $m\ge p >1$, $\bK = \bR$ or $\bC$, is fundamental for the study of the topological properties of  the map.  This has been first shown by Milnor \cite{Mi} in case of holomorphic functions $f:(\bC^{n},0) \rightarrow (\bC, 0)$, and had tremendous developments over several decades.\footnote{See also \cite{Le} for holomorphic function germs on a singular space germ $X$,  a different stream which we do not discuss here.} 
 
In \cite{Mi} it is proved that the Milnor fibration is independent of the local data.  More precisely, the local fibration existence statement includes the essential requirement that the locally trivial C$^\ity$-fibrations
\begin{equation}\label{eq:fib1}
   f_{|}: B^{2n}_{\e}\cap f^{-1}(D^{*}_{\eta}) \rightarrow D^{*}_{\eta}
\end{equation}
are independent, up to isotopy, of the small enough $\e \gg \eta >0$, where $B^{2n}_{\e}$ is the open ball at $0\in \bC^{n}$ of radius $\e$, and $D^{*}_{\eta}$ denotes the pointed disk of radius $\eta$ at $0\in \bC$.

   For holomorphic maps, Hamm \cite{Ha-icis} showed that if the map germ
 $G:(\bC^n,0)\to (\bC^k,0)$, $n>k>1$, is an \emph{isolated complete intersection singularity} (abbreviated ICIS\footnote{The ICIS condition amounts to the condition $\Sing G \cap G^{-1}(0) = \{0\}$. One of the richest sources of information on ICIS is  Looijenga's book \cite{Lo, Lo2} recently re-edited, and we refer to it for all the statements in the ICIS context.} in the following),  then $G$ defines a local fibration over the complement of its hypersurface discriminant:
 \begin{equation}\label{eq:fib2}
   G_{|}: B^{2n}_{\e}\cap G^{-1}(B^{2k}_{\eta}\m \Disc G) \rightarrow B^{2k}_{\eta}\m \Disc G
\end{equation}
  In particular this includes the independency of the choices of the small enough neighbourhoods, similar to Milnor's case evoked just before.
 
 In case of real map germs, as well as beyond the ICIS case in the holomorphic setting, 
  the existence of local fibrations is no more insured.  
Milnor  showed it in \cite[\S 11]{Mi} that analytic maps $G:(\bR^m,0)\to (\bR^p,0)$ define local tube fibrations in case they have isolated singularity $\Sing G = \{ 0\}$, but this is the only setting which needs no further conditions.\footnote{ 
In  \cite{Mi} Milnor also defines a local sphere fibration for holomorphic function germs,
and  subsequent study in the real setting has been concentrated to maps with isolated singularities (whenever such maps exist, see e.g. \cite{Lo0, AHSS}). We do not address this type of fibration here,  we only refer to \cite{ART2} for the newest developments.}
More recently there have been considered classes of singular maps $G$  for which  $\Disc G$ is the point $\{0\}$ only, thus zero dimensional,  and having non-isolated singularities in the fibre  over $\{0\}$,  cf   \cite{RSV,   dST0, Ma, dST1, ACT, DA, MS, PT} and many more others, but we do not intend to discuss more about this particular setting here. 

 Beyond the ICIS setting, the positive dimensional discriminant $\Disc G$ in local fibrations occurred before in some very special cases, e.g. \cite[Theorem 1.3]{ACT}, \cite[Prop. 2.1]{CGS}. 
 
 We address here  the possibility of defining \emph{locally trivial fibrations} in case of 
  real or complex map germs with positive dimensional discriminant $\Disc G$.
  In  this new setting the following problems have to be taken into account such that
the fibration problem can be well posed, and that the fibration becomes a local invariant:

 \medskip
\begin{enumerate}
  \item the local fibration must be independent of the small enough neighbourhood data, like in \eqref{eq:fib1} and \eqref{eq:fib2}. This does not come automatically for map germs outside the ICIS case;
  
\item   the image of the map germ $G$ may not be a neighbourhood of $0$ in $\bR^{p}$. Moreover, it may not be independent of the radius $\e$ of the ball $B^{m}_{\e}\subset \bR^{m}$, and thus the image of $G$ may not be well defined as a set germ in $(\bR^{p}, 0)$, see Definition \ref{d:nice};

 \item the discriminant of $G$ may not be well defined as a set germ. In case the image $G(\Sing G)$ of the singular locus is a set germ,  and when the image $\im G$ is a set germ too and has a boundary which contains the origin 0, then the discriminant $\Disc G$ should contain this boundary.

\end{enumerate}
 
%

\medskip

Here we show how all the above issues can be treated in the same time. 

We consider first the existence of a local fibration outside the discriminant. 
 Since we  need to rely on the image $\im G$ of  $G$ as a well-defined  set germ at $0$, as well as on the discriminant as 
 a well-defined germ of the image of the singular locus, we start in \S \ref{s:tame} the study of the delicate question whether the image of  an analytic map is well-defined as germ at the origin in the target\footnote{Recent progress \cite{JT} was made in classifying map germs with respect to this problem.}. We single out  classes of map germs $G$ such that $\im G$ and $\Disc G$ are well-defined as germs at $0$ and we call
them ``nice'' (Definition \ref {d:nice}). This category includes of course the complex ICIS.  
Within this framework,  we then give  an appropriate definition of the discriminant $\Disc G$ as the \emph{locus where the topology of the fibres may change}, Definition \ref{d:disc}.



 %

 In \S \ref{s:milnorhamm} we prove the existence result, Lemma \ref{full}, under the most general conditions. The locally trivial fibration provided by this lemma 
 in case of positive dimensional $\Disc G$  will be called
 \emph{Milnor-Hamm fibration}\footnote{We prefer not to call this a ``tube fibration'' since a large part from $B^{p}_{\delta}$ may be missing. We reserve ``singular tube fibration'' for our second type of fibration presented in \S \ref{s:sing}.}   in reference to the fibration structure defined in the 1970's by Hamm \cite{Ha-icis} in case of an ICIS.
 
 Comparing our local invariant to the existing literature, let us remark that \cite[Prop. 2.1]{CGS} states a fibration result in case of a positive dimensional discriminant, and another definition of a fibration in case of a positive dimensional discriminant was given in the very beginning of \cite[Section 2]{MS}. However these differ radically from our approach to local fibrations since they do not address the  above displayed key issues (a), (b) and (c).
 
 \medskip 
 
 We next prove that a weaker Thom condition, the \emph{$\partial$-Thom regularity}\footnote{As introduced in the 1990's, see Definition \ref{defT} and the discussion before it.}, is sufficient for the existence of  the Milnor-Hamm fibration.   We find classes of singularities which satisfy our conditions, for instance certain maps of type $f\bar g$ (see Theorem \ref{t:mainfbarg}, Proposition \ref{p3}).  We construct in \S \ref{examples} examples   where the map germ $G$ has Milnor-Hamm fibration without being  $\partial$-Thom regular, which is another new feature.
 

 \smallskip
 
The second aim of this paper is to introduce  the \emph{singular Milnor tube fibration} (cf Definition \ref{d:tube1}) which includes the fibres over the  discriminant. We give in \S \ref{s:sing} an existence condition which is more general than the Thom regularity condition,  and we find examples of singular Milnor tube fibrations either having Thom regularity or not having it. The later possibility has been recently discovered and rises interesting questions, see e.g. \cite[\S 5.1]{PT}.

The following very simple class shows several  aspects of the new fibrations that we introduce in this paper:

\begin{example}\label{ex:hh} 
Let $h:(\bC^n, 0) \to (\bC,0)$ be a holomorphic function germ. Consider the real valued function germ $H:(\bC^n, 0) \to (\bR,0)$ given by $H:=h\bar{h}$; this is clearly a nice map germ (Definition \ref {d:nice}). One has $\Sing H = V_H$ and one can show that the map $H$ is Thom regular at $V_H$.
By Proposition \ref{th:tube},  $H$ has a Milnor-Hamm fibration over $\bR\m \{0\}$ with two types of  fibres, one empty and one diffeomorphic to $B_{\e} \cap h^{-1}(S^{1}_{\eta})$, for some $0<\eta \ll \e$.

If we consider $H$ as a mixed map $f: \mathbb{C}^n \to \mathbb{C}$ with $f:=h\bar h$,   then $\Sing f = \mathbb{C}^n$ and $\Disc f=\im f = \bR_{\ge 0}$,  thus its Milnor-Hamm fibration over $\bC\m\Disc G$ has empty fibre.   Nevertheless,  its singular tube fibration (Definition \ref{d:tube1})  exists and has precisely the two types of  fibres of the map $H$.
 \end{example}
 
 As final remark, let us point out that with the same methods based on stratifications, particularly those in Section \ref{s:sing}, one can extend the framework to maps 
 $G:(X,0) \rightarrow (\bR^{p}, 0)$ where  $(X, 0) \subset (\bR^{N},0)$ is singular analytic. More generally,  one can work inside some o-minimal structure. We leave to the reader the technical details.

\section{Images, discriminants and nice map germs}\label{s:tame}

While working with polynomial or  analytic map germs,  and  with semialgebraic or  subanalytic sets, respectively,  the first obstruction to define a fibration is that the image of a set germ at 0 in the source must be a set germ at 0 in the target, and moreover,  that the discriminant of the map should be well-defined as a set germ.  This holds for function germs and  also for the complex ICIS maps,  but it is far from being true in general, neither  in the real setting, nor in the complex one.  
Such a negative phenomenon can be observed even for very simple polynomial maps, as follows:
\begin{example}\label{ex:notnicebasic}
 Let  $G : \bR^{2} \to \bR^{2}$, $G(x, z) =  (x, xz)$.  For  the 2-disks $D_t := \{|x|<t, |z|< t\}$ as basis of open neighbourhoods of 0 for $t>0$,
we get that the image $A_t := G(D_t)$ is the full angle with vertex at 0, having the horizontal axis as bisector,  and of slope $<t$.  Since the relations defining  $A_t$ depend of $t$, it means that the image of $G$ is not well-defined as a germ.
A similar behaviour happens over $\bC$ instead of $\bR$.
\end{example}

Let $A, A'\s \bR^{p}$ be  subsets containing the origin and let $(A,0)$ and $(A', 0)$ denote their germs at 0. We recall that one has the equality of set germs $(A,0) = (A',0)$ if and only if there exists some open ball $B_{\e}\s \bR^{p}$ centred at 0 and of radius $\e>0$ such that $A\cap B_{\e} = A'\cap B_{\e}$.

\begin{definition}\label{d:nice}
Let $G:(\bR^{m},0) \rightarrow (\bR^{p}, 0)$, $m\ge p >0$, be a continuous map germ.
 We say that the image  $G(K)$  of a set $K\subset \bR^{m}$ containing $0$  is a \emph{well-defined set germ}
 at $0\in \bR^{p}$ if for any open balls $B_{\e},  B_{\e'}$ centred at 0, with  $\e, \e' >0$, we have the equality of germs
 $(G(B_{\e}\cap K), 0) =  (G(B_{\e'}\cap K),0)$.
 
Whenever the images $\im G$ and  $G(\Sing G)$ are well-defined as set germs at 0,  we say that $G$ is a \emph{nice map germ}.

\end{definition}

\begin{remark}
In support of the above definition, let us point out that even if the image $\im G$ of a map $G$ is well-defined as a germ, the restriction of $G$ to some subset 
might be not.  This behaviour can be seen in the following example related to the above one,
namely let   $G:\bC^3\to\bC^2$, $G(x,y,z)=(x,z)$ and  $K:= \{(x,y,z) \mid z=xy\}\s \bC^3$. Then the image $G(K)$ is not well-defined as a set germ.
\end{remark}

We have first the following general result:
\begin{lemma}\label{l:tame}
Let $G:(\bR^{m},0) \rightarrow (\bR^{p}, 0)$, $m\ge p >0$ be an analytic map germ. 
\begin{enumerate}
\rm \item \it
 If $\Sing G \cap G^{-1}(0) \subsetneq  G^{-1}(0)$ then  $\im G$  is well-defined as a set germ.
\rm \item \it If $\Sing G \cap G^{-1}(0) = \{ 0\}\subsetneq G^{-1}(0)$ then  $G$ is nice.
 \end{enumerate}
\end{lemma}
\begin{proof} 
(a). Let $q\in  G^{-1}(0) \m \Sing G \not= \emptyset$ by hypothesis.  Then $G$ is a submersion on some small open
neighbourhood $N_{q}$ of $q$,  thus the restriction $G_{| N_{q}}$ is an open map, and therefore $\im G$ contains some open neighbourhood of the origin of the target. This shows the equality of set germs $(\im G, 0)= (\bR^{p}, 0)$.

\noindent
(b). Let us show that $G(\Sing G)$ is well-defined as set germ at $0\in \bR^{p}$.
By contradiction, if this is not the case, then there exist $\e >\e' >0$ and a sequence of points $p_{n}\in \bR^{p}$,
	$p_{n} \to 0$, $p_{n}\in G(B_{\e}\cap \Sing G)$, $p_{n}\not\in G(B_{\e'}\cap \Sing G)$ for all integers $n\gg 1$. Let  then $x_{n}\in B_{\e}\cap \Sing G$ with $G(x_{n}) = p_{n}$ and there is a subsequence
	$(x_{n_{k}})_{k\in \bN}$ which tends to some point $x$ in the closure $\bar B_{\e}\cap \Sing G$.
	We have $G(x) = \lim_{k\to \infty} G(x_{nk}) = \lim_{k\to \infty}p_{n_{k}}= 0$. Since $\Sing G\cap G^{-1}(0) = 0$, 
	the point $x$ must be 0.  But then we must have $x_{n_{k}}\in B_{\e'}\cap \Sing G$ for some $k\gg 1$, which implies that $p_{n_{k}}\in G(B_{\e'}\cap \Sing G)$ which is a contradiction to the assumptions about the sequence $(p_{n})_{n}$.
	
	Since the hypothesis in (b) is more particular than that in (a), it then follows in addition that $G$ is nice.
\end{proof}

\begin{remark}\label{r:K-det}\label{c2}
It is known that for analytic map germs $G:(\bR^m,0) \to (\bR^p,0)$, $m>p$, the condition $\Sing G \cap V_G = \{0\}$ is equivalent to  the finite $C^0$-$\mathcal{K}$-determinacy of $G$, see \cite{Wa,CB} for details.  Observe also that the condition in Lemma \ref{l:tame}(a) cannot be weakened, like for instance in the example  $f(x,y,z)=(x^2+y^2, (x^2+y^2)z)$ taken from \cite{Ha}. 
\end{remark}

\smallskip

We still assume that $G:(\bR^{m},0) \rightarrow (\bR^{p}, 0)$, $m\ge p >0$  is an analytic map germ. Whenever $\im G$ is well-defined as a set germ,
its boundary $\partial \overline{\im G} := \overline{\im G} \m \i(\im G)$ is a closed subanalytic  proper subset of $\bR^{p}$, where $\i A := \mathring{A}$ denotes the $p$-dimensional interior of a subanalytic set $A\subset \bR^{p}$ (hence it is empty whenever $\dim A <p$), and $\overline{A}$ denotes the closure of it. We consider here  $\partial \overline{\im G}$ as a set germ at $0\in \bR^{p}$; this is of course empty if (and only if)  the equality $(\im G, 0) = (\bR^{p}, 0)$ holds.

\begin{definition}\label{d:disc}
We call \emph{discriminant of a nice map germ $G$} the following set:
\begin{equation}\label{eq:disc}
\Disc G := \overline{G(\Sing G)} \cup \partial \overline{\im G}
\end{equation}
which is a closed subanalytic germ of dimension strictly less than $p$. 

All the above subsets are considered as germs at the respective origins, and they are well-defined since $G$ is nice. 
In particular, the complementary of $\Disc G$ in a neighbourhood of the origin is the disjoint union of finitely many open connected subanalytic sets, well-defined as germs at the origin. 
 \end{definition}
 
 We shall denote by $\gcd(f,g) = 1$ the 
greatest common divisor of two holomorphic function germs $f,g : (\bC^{n}, 0)\to (\bC, 0)$, thus well-defined only modulo units in the algebra of such function germs.
 

\begin{theorem}\label{p:fbarg-tame}
Let $f, g: (\bC^{n}, 0)\to (\bC, 0)$, $n>1$,  be holomorphic germs such that $\gcd(f,g) = 1$. 
Then 
$f\bar g : (\bC^{n}, 0)\to (\bC, 0)$ is a nice map germ.  
\end{theorem}
\begin{proof} 
By our assumption we have $\ord_{0} f >0$ and $\ord_{0} g>0$.  Let us consider 
 some  representatives $f,g : B_{\e} \to \bC$ of the holomorphic germs, well-defined on some small ball $B_{\e}$. 
For any $\alpha \in S^{1}\subset \bC$
 the set $A_{\alpha} := \{ f - \alpha g =0\}\subset B_{\e}$ is a complex hypersurface germ at the origin $0\in \bC^{n}$,
 hence of dimension $n-1$. The image of the restriction of $f\bar g$ to $A_{\alpha}$ is $\im \alpha |g_{|A_{\alpha}}|^{2}$,
 and we claim that the restriction $g_{|A_{\alpha}}$ is not constant $=0$.
 
If not the case, then $g_{| A_{\alpha}} \equiv 0$ implies that there is some holomorphic germ $v$ of positive order such that $v$ divides both $g$ and $f - \alpha g =0$, so that its zero locus contains $A_{\alpha}$. This further implies that $v$ divides $f$, contradicting the irreducibility of the fraction $f/g$.
 
 We have thus shown that $\im g_{| A_{\alpha}}$ contains some small open disk $D_{r(\alpha)}$ centred at 0 of positive radius $r(\alpha)$, for any $\alpha \in S^{1}$.  The radius $r(\alpha)$ is  a lower semi-continuos function of $\alpha$ in the following sense:
 let $\alpha\in S^{1}$ be fixed;  for any $0< d < r(\alpha)$, there is a small neighbourhood $N_{\alpha}\s S^{1}$ of $\alpha$ 
 such that if $\alpha' \in N_{\alpha}$ then $\im g_{| A_{\alpha'}}$ contains some open disc $D_{d}$ centred at 0 of radius $d$.
 Indeed, for  $d<d'<r(\alpha)$, there exists some point $x_{\alpha}\in B_{\e}\cap (g_{| A_{\alpha}})^{-1}(\partial \bar D_{d'})$. Then the continuous function $g$ maps some small enough neighbourhood $M_{x_{\alpha}}\s B_{\e}$ of $x_{\alpha}$ to the complement of the disk $D_{d}$. On the other hand,  $M_{x_{\alpha}}$ is intersected by all hypersurfaces $A_{\alpha'}$
 for $\alpha'$ in some small enough neighbourhood of $\alpha$.
 
  Now since $S^{1}$ is compact, it  follows that there is $\rho>0$  such that $r(\alpha) > \rho$ for any $\alpha \in S^{1}$. 
 Altogether  this shows  that $\im (f\bar g)_{|B}$
 contains an open neighbourhood of $0\in \bC$, and this implies that the image of $f\bar g$ is well-defined as germ.

Let us show that the discriminant $\Disc f\bar g$ is also well defined as a set germ at $0$.
 First we remark that the discriminant $\Disc (f,g)$ of the map $(f,g)$ is well defined as a set germ at the origin of $\bC^{2}$, namely it is either $0$ or a the germ of a complex curve which may have several irreducible components. This follows from the fact that the restriction of $(f,g)$ to $\Sing (f,g)$ is either constant or locally a complex curve (i.e. at points where the rank of this restriction is 1).
 
Next, by considering the decomposition of the map $f\bar g$ into $(f,g)$ followed by $u\bar v: \bC^{2} \to \bC$ (see \cite{PT}),
one finds that the discriminant of   $f\bar g$  is the image by $u\bar v$ of the critical locus of all the restrictions
$u\bar v_{|C}: (C,0) \to (\bC, 0)$ to irreducible components $(C,0)\subset \Disc (f,g)$. By considering the Puiseux parametrisation $p_{C}: (\bC,0) \to (C,0)$, the  critical locus of $u\bar v_{|C} \circ p_{C}$ is defined as the solution of real analytic equations, and therefore it is a finite collection of real analytic curves. The images by $u\bar v_{|C} \circ p_{C}$
of these real analytic curves is a finite union of real semi-analytic curves (by \L ojasiewicz' result in dimension 1) which are obviously well defined as set germs at 0. It then follows that $\Disc f\bar g$ is well defined as a set germ, and it is either $0$ or a finite union of real semi-analytic curve germs.

This completes the proof  that $f\bar g$ is a nice map germ.  
\end{proof}

\section{Milnor-Hamm fibration}\label{s:milnorhamm}

We still work with a non-constant nice analytic map germ $G:(\bR^{m},0) \rightarrow (\bR^{p}, 0)$, $m\ge p>0$.
We denote $V_{G} := G^{-1}(0)$ and $B^{k}_\eta$ signifies an open ball of radius $\eta$ centred at $0\in \bR^{k}$, for some positive integer $k$.

After \eqref{eq:disc}, the open subset $B^{p}_\eta \m \Disc G$ is the disjoint union $\sqcup_{i}\cC_{i}$ of finitely many open subanalytic sets considered as germs at the origin, and we have by definition:
\begin{equation}\label{eq:cond}
  \im G \cap \cC_{i} \not= \emptyset \  \Leftrightarrow \ \im G \supset \cC_{i}.
\end{equation}

We would like to have 
 a locally trivial C$^{\ity}$-fibration with well-defined fibre over each such component $\cC_{i}$.  

\begin{definition}\label{d:Hamm}
Let $G:(\bR^{m},0) \rightarrow (\bR^{p}, 0)$, $m\ge p>0$,  be a non-constant nice analytic map germ.
We say that $G$ has {\em  Milnor-Hamm fibration} 
if, for any $\e > 0$ small enough, there exists  $0<\eta \ll \e$ such that the restriction:
\begin{equation}\label{eq:tube}
 G_| :  B^{m}_{\e} \cap G^{-1}( B^{p}_\eta \m \Disc G) \to  B^{p}_\eta \m \Disc G
\end{equation}
 is a locally trivial fibration over each connected component $\cC_{i}\subset B^{p}_\eta \m \Disc G$, such that it is independent of the choices of $\e$ and $\eta$ up to diffeomorphisms. 
 \end{definition}

It follows that, if it exists,  this fibration has nonempty fibres only if $\dim \im G =p$. The case $\dim \im G <p$ will be included in a more general treatment in \S \ref{s:sing}.   



Let $U \subset \bR^m$ be an open set such that $0\in U$,  and let $\rho:U \to \bR_{\ge 0}$ be a non-negative proper function which defines the origin, for instance the Euclidean distance. 
	The transversality of the fibres of a map $G$ to the levels of $\rho$ is called \textit{$\rho$-regularity} and it is a sufficient condition for the existence of  locally trivial fibrations. It was used in the local (stratified) setting by Thom, Milnor, Mather, Looijenga, Bekka, e.g.  \cite{Th1,Th2,Mi,Lo,Be}  and more recently \cite{ACT, dST1}, as well as at infinity, e.g. \cite{NZ,Ti2,Ti3,ACT-inf,DRT}, but also in many other recent papers and under different names.  
We consider the following definition:	

\begin{definition}\label{d:M}
	Let 	$G:(\bR^m, 0) \to (\bR^p,0)$ be a non-constant analytic map germ, $m\ge p >0$. We call:
	\[M(G):=\left\lbrace x \in U \mid \rho \not\pitchfork_x G \right\rbrace  \]
	the set of \textit{$\rho$-nonregular points} of $G$, or \emph{the Milnor set of $G$},
	where  $\rho$ denotes here the Euclidean distance function.
\end{definition}
 

We then consider the following condition:
\begin{equation}\label{eq:main}
 \overline{M(G)\m G^{-1}(\Disc G)}\cap V_G \subseteq \{0\}
\end{equation}
where the analytic closure of the subanalytic set $M(G)\m G^{-1}(\Disc G)$ is considered as a set germ at the origin.
Condition  \eqref{eq:main} is a direct extension of the condition used in  \cite{dST0, dST1, Ma, ACT} for the case $\Disc G =  \{0\}$,
where it was shown that it
insures the existence of a locally trivial  fibration  called \emph{Milnor tube fibration}\footnote{We send to \S \ref{s:sing} for the extension of this fibration.}. The basic result on the existence of Milnor-Hamm fibrations is the following:

\begin{lemma} \label{full}
	Let $G:(\bR^m, 0) \to (\bR^p,0)$ be a non-constant nice analytic map germ, $m> p >0$.  If $G$ satisfies condition \eqref{eq:main}, then $G$ has a Milnor-Hamm fibration \eqref{eq:tube}.
\end{lemma}
\begin{proof} 
 The map $G$ satisfies condition \eqref{eq:main}
if and only if there exists $\e_0>0$ such that, for any  $0<\e <\e_0$, there exists $\eta$, $0<\eta \ll \e$, such that the restriction map
\begin{equation}\label{rest1}
G_{|}: S_{\e}^{m-1}\cap G^{-1}(B^{p}_{\eta} \m \Disc G) \to B^{p}_{\eta} \m \Disc G
\end{equation}
is a smooth submersion and moreover that the map
$G_{|}:\overline{B^m_\e} \cap G^{-1}(B^{p}_\eta \m \Disc G) \to B^{p}_\eta \m \Disc G$
	is a submersion on a manifold with boundary, and  it is proper. We may therefore apply Ehresmann's Theorem (see for instance \cite{JM}) to conclude to the existence of a locally trivial fibration \eqref{eq:tube} over every connected component of $B^{p}_{\eta} \m \Disc G$ (with empty or non-empty fibre). 	
\end{proof}

\begin{remark}\label{r:notbad}
A map germ $G$ such that  $(M(G), 0)= (\bR^{m}, 0)$ can still have a  Milnor-Hamm fibration, here is an example:
 $G := (x^{2}+ y^{2}, z^{2}+ w^{2}) : \bR^{4}\to \bR^{2}$. 
The  map $G$ is nice, with
	$\im G = \bR_{\ge 0} \times \bR_{\ge 0}$, $\Disc G = \bR_{\ge 0}\times \{0\} \cup \{0\} \times \bR_{\ge 0}$.
We have $V_{G} =\{0\}$ and  $(M(G),0)= (\bR^{4},0)$, compare to Theorem \ref{c4} below.

Nevertheless, $G$ satisfies condition \eqref{eq:main} and  the Milnor-Hamm fibration exists. Its fibre over the  positive quadrant is compact, diffeomorphic to a torus $S^{1}\times S^{1}$ which does not intersect the Milnor-Hamm sphere $S_{\e}^{m-1}$ in \eqref{rest1}. 
\end{remark}

The next statement represents a real counterpart of Hamm's result \cite{Ha-icis} that a holomorphic map which defines an ICIS has a locally trivial fibration over the complement of the discriminant:
\begin{theorem}\label{c4}
		Let $G:(\bR^m, 0) \to (\bR^p,0)$ be a analytic map germ. If  $V_G  \cap \Sing G  = \{0\}$ and  $\dim V_{G}>0$,		 then $G$ is nice and has a Milnor-Hamm fibration \eqref{eq:tube}. 
 
\end{theorem}

\begin{proof}
	Since  $V_G  \cap \Sing G  = \{0\}$ and $\dim V_{G}>0$, we may apply Lemma \ref{l:tame}(b)
	and conclude that $G$ is nice.
		
	
Next, since transversality is an open condition, it follows that for each $\e>0$ one can find $0<\eta_{0} \ll \e$, such that for any $0<\eta \le \eta_{0}$ the restriction map 
\begin{equation}\label{fibonSp}
		G_{|}:S^{m-1}_{\e} \cap G^{-1}(B_{\eta}^{p})\to B_{\eta}^{p}
		\end{equation}
 is a smooth submersion and therefore Lemma \ref{full} applies. 	
 \end{proof}
 


\smallskip

\begin{example}  
	The map $G: \bR^n \to \bR^2$,  $G(x_1,\ldots,x_n)=(x_{1},x_{2}^{2}+\cdots+x_{n-1}^{2} - x_{n}^{2})$ is nice. 
	One has $V_G=\{x_{1}=0\} \cap \{x_{2}^{2}+\cdots+x_{n-1}^{2} - x_{n}^{2}=0\}$ and  $\Sing G = \{x_{2}=\cdots= x_{n}=0\}$ is the $x_1$-axis. The condition $V_G  \cap\Sing G = \{0\}$ is satisfied and the Milnor-Hamm fibration \eqref{eq:tube} exists by Theorem \ref{c4}.
\end{example}

\section{Partial Thom regularity condition and  Milnor-Hamm fibration} \label{s:partial-thom}


Given some stratification of a neighbourhood of $0 \in \bR^{m}$ we recall, after \cite{GLPW, JM},  that a stratum $A$ is \textit{Thom regular} over a stratum $B\subset \bar A\m A$ at $x\in B$ relative to $G$ (or, equivalently,  that the pair $(A,B)$ satisfies the Thom (a$_G$)-regularity condition at $x$),  if the following condition holds: for any $\{x_n\}_{n\in \bN}\subset A$ such that  $x_n\to x$, if 
$T_{x_n}(G_{|A})$ converges, when $n\to \ity$,  to a limit $H$
in the appropriate Grassmann bundle,  then  $T_{x}(G_{|B})\subset H$. 
		
The notion of $\partial$-Thom regularity has been introduced in \cite[Def. 2.1]{Ti1} for analytic function germs, and 
in  \cite[A.1.1.1]{Ti3} and \cite[\S 6]{DRT} as a regularity condition in the neighbourhood of infinity.  Here we give the map germ version. The difference to the classical Thom regularity is that here we do not ask all possible conditions between couples of strata $(A,B)$ for $A$ outside $V_G$ and $B$ inside $V_G$,  but only those for  $A$ equal to the single stratum $B^{m}_{\e}\m G^{-1}(\Disc G)$.

\begin{definition}\label{defT} 
Let $G:(\bR^m, 0) \to (\bR^p,0)$ be a non-constant  analytic map germ. We say that $G$ is \textit{$\partial$-Thom regular at $V_G$} if there is a a ball $B^{m}_{\e}$ centred at $0\in \bR^{m}$ and a Whitney (a)  semi-analytic stratification $\bH=\{H_{\alpha}\}$ of $B^{m}_{\e}\cap V_G$ such that, for any stratum $H_\alpha$, the pair $(B^{m}_{\e}\m G^{-1}(\Disc G), H_\alpha)$  satisfies the Thom (a$_G$)-regularity condition. 

We then also say that such a stratification is a \emph{$\partial$-Thom  $(a_{G})$-stratification of $V_{G}$}.
\end{definition}

Let us remark that Thom regularity at $V_G$ obviously  implies $\partial$-Thom regularity at $V_G$.
 The following result shows that the weaker regularity condition is sufficient to insure the existence of the Milnor-Hamm fibration.

\begin{proposition}\label{th:tube}
	Let $G:(\bR^m, 0) \to (\bR^p,0)$ be a non-constant  nice analytic map germ. If $G$ is $\partial$-Thom regular at 
	$V_G$,  then $G$ has a Milnor-Hamm fibration \eqref{eq:tube}.
\end{proposition} 
\begin{proof}

The $\partial$-Thom regularity implies that for any $\e>0$ small enough, there exists a neighborhood $N_\e$ of $S^{m-1}_{\e} \cap  V_G$ in $S^{m-1}_{\e}$ such that for any $x \in N_\e \m G^{-1}(\Disc G)$ one has $S^{m-1}_{\e} \pitchfork_x G^{-1}(G(x))$. 
	

	
	It follows that there is  $\eta_\e >0$ such that $S^{m-1}_{\e} \cap G^{-1}(B^{m}_{\eta_{\e}}\m \Disc G) \subset N_\e$. Then the restriction \eqref{rest1} is a submersion and consequently $G$  satisfies condition \eqref{eq:main} and one may apply Lemma \ref{full}.
\end{proof}

\subsection*{Classes of maps with $\partial$-Thom regularity}
We single out several classes of maps which have Milnor-Hamm fibration because of being $\partial$-Thom regular, thus to which Proposition   \ref{th:tube} applies.
For that, let us formulate  the definition of the \emph{$\partial$-Thom regularity set}, in analogy to that of the Thom regularity set of \cite[\S3]{PT}. Referring to Definition \ref{defT}, let 
\begin{equation}\label{eq:Thom}
\partial \NT_{G}:=\left\lbrace x \in V_{G} \, \left| \begin{tabular}{l}
\textit{there is no $\partial$-Thom  $(a_{G})$-stratification of $V_{G}$ such}   \\
\textit{that $x$ belongs to a positive dimensional stratum.}  
\end{tabular}\right.
 \right\rbrace 
\end{equation}


The definition of the \emph{Thom regularity set}  $\NT_{G}$  of  \cite[\S3]{PT}
is obtained by just replacing  ``$\partial$-Thom'' by ``Thom''  in the above definition. 

Let us first consider maps $G= (f,g)$, where $f, g : (\mathbb{C}^n,0) \to (\mathbb{C},0)$, $n> 1$, are non-constant complex analytic function germs, and let  $f\bar{g}:(\mathbb{C}^n, 0) \to (\mathbb{C}, 0)$. The latter are a special class of mixed functions. 
   
 \begin{theorem}\label{c3} \label{t:mainfbarg}
 Let $f, g : (\mathbb{C}^n,0) \to (\mathbb{C},0)$, $n > 1$, be  non-constant holomorphic functions.
Then:
\begin{enumerate}
\rm \item  \it If $(f,g)$ defines an ICIS then the map germ $f\bar{g}$ is nice,  Thom regular and has a Milnor-Hamm fibration. 
\rm \item  \it If  $\gcd(f,g)=1$ and the map $(f,g)$ is Thom regular (i.e. $\NT_{(f,g)}\s \{0\}$), then  the map $f\bar{g}$ has a Milnor-Hamm fibration.
\rm \item  \it  
  If   $\gcd(f,g)=1$,  if $\Disc (f,g)$ consists of only lines (for instance if $f$ and $g$ are homogeneous of the same degree), and if $(f,g)$ is $\partial$-Thom regular, 
then the map $f\bar{g}$  is $\partial$-Thom regular and has a Milnor-Hamm fibration.
\end{enumerate}
\end{theorem}

\begin{remark}\label{r:pichonseademistake}
 There are some well-known attempts to endowing $f\bar g$ with a Thom regular stratification.   Pichon and Seade stated and proved a general theorem in $n>2$ variables. As this was pointed out to be false  \cite{Pa}, \cite{Ti4},  in the subsequent erratum \cite{PS2} the authors presented Parusinski's counterexample \cite{Pa}, \cite{Ti4}, \cite[p. 827]{ACT}, \cite{Oka4}. Even in $n=2$ variables, the proof  of the statement  published earlier \cite[Proposition 1.4]{PS} has the same problem since it tacitly assumes that the normal space at a point of some fibre of $f\bar g$ is generated by a single vector.  
 
 The status of the existence of a  Thom regular stratification for $f\bar g$ was established recently in  \cite[Theorem 3.1]{PT} (see \cite{PT2} for the rectified statement of this theorem), based on different grounds, and for a large class of  maps $f\bar g$ in $n\ge 2$ variables.   In particular the proof  of  \cite[Theorem 3.1]{PT} shows that 
  the statement of the above discussed result  in 2 variables \cite[Proposition 1.4]{PS} is nevertheless correct.
\end{remark}

\begin{proof}[Proof of Theorem \ref{t:mainfbarg}]
In the following we use the proof of \cite[Theorem 3.1]{PT}. That proof starts with the hypothesis $\Disc f\bar g \subset \{0 \}$ and therefore remains unchanged after the rectification  \cite{PT2}\footnote{The corrigendum \cite{PT2} rectifies the erroneous criterion for the inclusion $\Disc f\bar g \subset \{0 \}$ given in \cite{PT}, as follows:
``\emph{$\Disc f\bar g \subset \{0 \}$  iff  the discriminant $\Disc (f,g)$ of the map $(f,g)$ contains 
only curve components which are tangent to the coordinate axes}''.} of the statement of \cite[Theorem 3.1]{PT}.

\smallskip
\noindent
(a). If $(f,g)$ defines an ICIS then there is the equality of germs $[\im (f,g)]_0 = [\bC^{2}]_{0}$, and $\Disc (f,g)$ is a well-defined curve germ in $(\bC^{2}, 0)$, see e.g. \cite{Lo, Lo2}. In particular it follows that $(f,g)$ is a nice map. It then follows from the proof of Theorem \ref{p:fbarg-tame} that $f\bar{g}$ is a nice map germ. 

In order to prove the $\partial$-Thom regularity of $f\bar{g}$, we proceed as follows. 
By Step 1 of the proof of \cite[Theorem 3.1]{PT},  we have $\NT_{(f,g)}\s V_{(f,g)}$. Since $\Sing (f,g) \cap V_{(f,g)} = \{0\}$, it follows that $(f,g)$ is Thom regular. One remarks that the Thom $(a_{(f,g)})$-regularity of the pair of strata $(B_{\e}\m \Sing  (f,g),  V_{(f,g)}\m \{0\})$ implies the Thom $(a_{(f\bar g)})$-regularity of the same pair; one can find all details 
in Step 2 of the proof of \cite[Theorem 3.1]{PT} where this claim is showed. In our simple situation, this just means that $f\bar g$ is $\partial$-Thom regular.

One may now apply Proposition \ref{th:tube} to conclude that the Milnor-Hamm fibration exists.

\smallskip
\noindent
(b). If  $\gcd(f,g)=1$, then it follows from Theorem \ref{p:fbarg-tame} that the maps $f\bar g$ is nice.

In order to prove the $\partial$-Thom regularity of $f\bar g$ it is only necessary to check the Thom regularity condition outside the inverse image of the discriminant set $\Disc f\bar{g}$.  This avoids the line components of the discriminant of $(f,g)$ but not the other components of the same  $\Disc (f, g)$. Nevertheless, the 
Step 2 of the proof of  \cite[Theorem 3.1]{PT} takes into account the components of the singular set $\Sing (f,g)$ which do not have line images. It therefore appears as a result of this proof  that the Thom $(a_{(f,g)})$-regularity condition, including along the components of  the singular set $\Sing (f,g)$, implies the
$\partial$-Thom $(a_{(f\bar g)})$-regularity.

Lastly we may apply the same  Proposition \ref{th:tube} for getting the Milnor-Hamm fibration.

\smallskip
\noindent
(c). Referring to  the proof of Theorem \ref{p:fbarg-tame}, let
 us observe that if $\Disc f\bar{g}$ contains a real half-line, then its inverse image by $u\bar v$ is a line component of 
$\Disc (f, g)$. Therefore the proof of (c) can be extracted from the proof of (b) by remarking that since there are  only line components of  $\Disc (f, g)$
we only deal with the Thom condition between strata outside the singular locus of $(f,g)$, which means the we only need the  $\partial$-Thom regularity of $(f,g)$ and not its full Thom regularity.
\end{proof}

While we do not know if the following statement  is true for the  stronger Thom regularity,  we can prove it for the $\partial$-Thom regularity, which is enough for our purpose, in view of Proposition \ref{th:tube}.
\begin{proposition}\label{p3}
Let $f:(\bR^{m},0)\to (\bR^{p},0)$ and $g:(\bR^{n},0)\to (\bR^{k},0)$ be nice analytic maps, in separate variables,  and satisfying the  $\partial$-Thom regularity condition at $V_{f} \not= \{0\}$ and $V_{g} \not= \{0\}$, respectively. Then the map $G:= (f,g)$ is nice and $\partial$-Thom regular at $V_G$.  


In particular, if $f$ and $g$ are holomorphic functions in separate variables, then
 $f\bar{g}$ is  $\partial$-Thom regular and thus has a Milnor-Hamm fibration.
\end{proposition}
\begin{proof} 
The niceness matter is an easy exercise which we may safely leave to the reader.

Next, the fibre $G^{-1}(c_{1},c_{2})\subset  \bR^{m}\times \bR^{n}$ of the map $G:=(f,g):(\bR^{m}\times \bR^{n},0)\to (\bR^{p}\times \bR^{k},0)$ is the product $f^{-1}(c_{1})\times g^{-1}(c_{2})$. If $\mathcal{S}_{f}$ and $\mathcal{S}_{g}$ are $\partial$-Thom regular stratifications of $f$ and $g$, respectively,  then the product stratification $\mathcal{S}_{f}\times \mathcal{S}_{g}$ on $(\bR^{m}\times \bR^{n},0)$ is a  $\partial$-Thom regular stratification of $V_G$.  The proof is also standard and we may safely skip it.

 Our last claim follows from Theorem \ref{t:mainfbarg}(c) since the  holomorphic functions  $f$ and $g$ 
 satisfy the hypotheses of that statement, as one can easily check. 
\end{proof}



\begin{example}
Let  $G:(\mathbb{C}^4,0) \to (\mathbb{C}^2,0)$ given by $G(z_1,z_2,z_3,z_4)=(z_1 \bar{z}_2,z_3 \bar{z}_4)$. By Proposition \ref{p3} each component of $G$ is a $\partial$-Thom regular function. Since the map $G$ is in separate variables,  Proposition \ref{p3} yields that $G$ is $\partial$-Thom regular as well. By Lemma \ref{full} the Milnor-Hamm fibration \eqref{eq:tube} exists for $G$. One has $V_G \cap \Sing G \supsetneq \{0\}$ and  $\Disc G = \{(z,0)\in \mathbb{C}^2\,|\,z\in \mathbb{C}\}\cup \{(0,w)\in \mathbb{C}^2\,|\,w\in \mathbb{C}\}$.

\end{example}

\begin{example} 
We construct a map germ $G:=(f,g):(\bR^{3}\times \bR^{4},0)\to (\bR^{4},0)$ as follows.
Let  $f:(\bR^3,0) \to (\bR^2,0)$, $f(x,y,z)=(y^4-z^2x^2-x^4,xy).$ One has that $V_{f} = \Sing f=\{x=y=0\}$, and $\im f = \bR^{2}$, thus $\Disc f = \{0\}$, in particular the map is nice.  It was proved in  \cite[Example 5.3]{ACT} that $f$ is Thom regular. 

Let  $g:=(g_1,g_2):(\bR^4,0) \to (\bR^2,0)$ be defined by:
\[\left\lbrace \begin{array}{ccl}
g_1(a,b,c,d)&= -2d^3ab-3d^2a^2c+3d^2b^2c+6dabc^2+a^2c^3-b^2c^3 &   \medskip\\
g_2(a,b,c,d)&= d^3a^2-d^3b^2-6d^2abc-3da^2c^2+3db^2c^2+2abc^3. &  \medskip
\end{array}\right.\] 

By computing the gradients, we get $\left\langle \nabla g_1, \nabla g_2 \right\rangle =0$
and $\|\nabla g_1 \|^2 =(a^2+b^2)(d^2+c^2)^2(4d^2+9a^2+9b^2+4c^2)= \|\nabla g_2 \|^2$, which shows that $g$ is a \emph{simple \L-map} after \cite[Definition 3.5]{Ma}, and in particular has isolated critical value.  Applying \cite[Theorem 2.12, Lemma 5.5]{Ma}, one obtains that  $g$ is Thom regular at $V_{g}$. In particular  its image covers an open subset containing the origin, hence we get that $\Disc g =\{ 0\}$, and that $g$ is nice.   

It then follows by Proposition \ref{p3} that  $G$  is nice and is $\partial$-Thom regular and hence it has a Milnor-Hamm fibration.  Since $\Disc G=\{0\}\times \bR^{2} \cup \bR^{2} \times \{0\}$, the complement of $\Disc G$ is connected, hence there is a unique fibre.  
\end{example}

\section{Constructing examples}\label{examples}
The first example with Milnor tube fibration but without Thom regularity occurred in case $\Disc G = \{0\}$ in \cite{Ti4, ACT}; see also \cite{Oka4, PT} and \cite{Ha}.
We show here  some examples in our new setting  $\Disc G \supsetneq \{0\}$, based on the  following:

\begin{lemma}\label{p4}\label{r:coupleId}
	Let $f:(\bR^m,0) \to (\bR^{p},0)$ be an analytic map germ and let $g:(\bR^n,0) \to (\bR^n,0)$ be the identity function $g(y)=y$. 	Consider the analytic map germ $G:=(f,g):(\bR^{m}\times \bR^n,0) \to (\bR^{p}\times \bR^n,0)$. Then:
\begin{enumerate}
\rm \item \it   $f$ is $\partial$-Thom regular if and only if $G$ is $\partial$-Thom regular.

\rm \item \it  $f$ satisfies  condition \eqref{eq:main} if and only if $G$ satisfies  condition \eqref{eq:main}.
\end{enumerate}
\end{lemma}
\begin{proof}
(a) is a simple exercise. Let us prove (b).
	Assume that $f$ satisfies condition $\eqref{eq:main}$. Let $p_0=(x_0,y_0)\in \overline{M(G)\m G^{-1}(\Disc G)}\cap V_G$. One has a sequence of points $p_n=(x_n,y_n)\to p_0$ such that $p_n \in M(G)\m G^{-1}(\Disc G)=M(f)\times \bR^n \m f^{-1}(\Disc f)\times\bR^n$ and $p_0\in V_G=V_f\times\{0\}$. Thus, $y_0=0$, $x_0\in V_f$ and $x_n\in M(f) \m f^{-1}(\Disc f)$. Consequently, $x_0\in \overline{M(f)\m f^{-1}(\Disc f)}\cap V_f \subset \{0\}$ which implies  $p_0 = (0,0)$ and $G$ satisfies the condition \eqref{eq:main}. 
	
	
	Assume that $\overline{M(G)\m G^{-1}(\Disc G)}\cap V_{G} \subseteq \{0\}$. One has that $$ \left(M(f)\m f^{-1}(\Disc f)\right)\times \{0\}  \subset M(G)\m G^{-1}(\Disc G).$$
	For any $ x_0 \in \overline{M(f) \m f^{-1}(\Disc f)} \cap V_f$  there exists a sequence $ x_n \subset M(f) \m f^{-1}(\Disc f) $ such that $ x_n \to x_0$. Hence, $ (x_n , 0) \in M(G)\m G^{-1}(\Disc G)$ for all $n \in \mathbb{N}$ and the limit point $ (x_0, 0) \in \overline{M(G)\m G^{-1}(\Disc G)} \cap V_{G}$. We conclude that $ x_0 = 0$, which implies that $f$ satisfies the condition \eqref{eq:main}.
\end{proof}

\begin{example} \label{ex:1couple}
	Let $G=(f,g)=(\bR^4,0) \to (\bR^3,0)$, where $f(x,y,z)=(x,y(x^2+y^2)+xz^2)$ and $g(w)=w$.   After \cite[Example 1.4.9]{Ha},  the map $f$ has isolated critical value and satisfies the condition \eqref{eq:main}. We may easily check that $f$ is an open map at 0, and that $f$ is nice. By Lemma \ref{p4},  $G$ satisfies the condition \eqref{eq:main}, therefore $G$ has a Milnor-Hamm fibration.   Note that $\Disc G = \{(0,0)\} \times \bR$, and that $G$ is nice.
	
On the other hand, \cite[Example 2.3.9]{Ha} shows that the map $f$ is not Thom regular. 
Therefore, by  Lemma \ref{r:coupleId}(a), the map $G$ cannot be $\partial$-Thom regular.	
\end{example}

\begin{example}  \label{ex:2couple}
Let $F$ be one of the mixed functions:
\begin{enumerate}
\item [1)] $F_{1}(x,y)=xy\bar{x}$ from \cite{ACT},
\item [2)] $F_{2}(x,y,z)=(x+z^k)\bar{x}y$ for a fixed $k\geq 2$ from \cite{PT}, 
\item [3)] $F_{3}(w_1,\ldots,w_n)=w_1\left(\sum_{j=1}^{k}|w_j|^{2a_{j}} - \sum_{t=k+1}^{n}|w_t|^{2a_{t}}\right)$ from \cite{Oka4}. 
\end{enumerate}
They are all  polar weighted-homogeneous and thus, by \cite[Theorem 1.4]{ACT}, one obtains that  $\Disc F = \{ 0\}$ and that $F$  is nice and has Milnor tube fibration.  It was also proved in the respective papers that $F$ is not Thom regular.

Let then $G:=(F,g)$, where $g(v)=v$ and note that $\Disc G=\{0\} \times \mathbb{C}$. By Lemma \ref{p4} the map $G$ satisfies condition \eqref{eq:main} and therefore, by Lemma \ref{full}, $G$ has Milnor-Hamm  fibration. However, $G$ is not $\partial$-Thom regular, cf Lemma \ref{r:coupleId}(a). 
\end{example}

\begin{example}
	Let $G=(f,g):(\bR^6 \times \bR^n,0) \to (\bR^4\times \bR^n,0) $, where $g:(\bR^n,0) \to (\bR^n,0))$ is the identity map germ and 
	$ f:=(f_1,f_2,f_3,f_4):(\bR^6,0) \to (\bR^4,0) $ where 
	
	$$ \left\lbrace \begin{array}{ccl}
	f_1(x,y,z,w,a,b)&= & yw+xz  \medskip\\
	f_2(x,y,z,w,a,b)&= & xw-yz \medskip\\
	f_3(x,y,z,w,a,b)&= & ax+by \medskip\\
	f_4(x,y,z,w,a,b)&= & ay-bx \medskip\\
	
	\end{array}\right. $$
	
	One has $V_f  =  \{{x}^{2}+{y}^{2}=0\} \cup \{{a}^{2}+{b}^{2}+{z}^{2}+{w}^{2}=0\}$, 
	$\Sing f  = \{{x}^{2}+{y}^{2}=0\}$.
One can show that $f$ is a \emph{simple \L-map}, and therefore it is nice.
	The Milnor set is $M(f)  =    \{{x}^{2}+{y}^{2}=0\} \cup \{{a}^{2}+{b}^{2}+{z}^{2}+{w}^{2} -{x}^{2}-{y}^{2}=0\}$.	
 Note that $M(G) = M(f)\times \bR^{n}$. Since $f$ has  isolated critical value and satisfies condition \eqref{eq:main}, it follows by Lemma \ref{p4} that $G$ satisfies condition \eqref{eq:main} and therefore has a Milnor-Hamm fibration by Lemma \ref{full}. Its discriminant set $\Disc G =\{0\}\times \bR^n$ has positive dimension.
\end{example}

%
\section{The singular Milnor tube fibration}\label{s:sing}

We extend here the  definition of Milnor-Hamm fibration to that of a  ``tube fibration''  by including the discriminant in the picture.

\begin{definition}\label{d:stratif}
	Let  $G:(\bR^{m},0) \rightarrow (\bR^{p}, 0)$ be a non-constant  analytic map germ, $m\ge p >1$.
	Let $G_{\e}: B^{m}_{\e} \to \im G_{\e}$ denote the restriction of $G$ to a small ball. 
	
By the stratification theory, see e.g. \cite{GLPW, GM},  there exist  locally finite subanalytic Whitney stratifications $(\bW, \bS)$  of the source of $G_{\e}$ and of  its target, respectively, such that $\overline{\im G_{\e}}$ is a union of strata, that $\Disc G_{\e}$ is a union of strata,  and that $G_{\e}$ is  a stratified submersion.  In particular every stratum is a nonsingular, open and connected subanalytic set at the respective origin, and morever:
\begin{enumerate}
		\item [(i)] The image by $G_{\e}$ of a stratum of  $\bW$ is  a single stratum of $\bS$, 
		\item[(ii)] The restriction $G_{|}:W_\alpha \to S_{\beta}$ is a submersion, where $W_{\alpha} \in\bW$, and $S_{\beta}\in \bS$.
	\end{enumerate}
	One calls  $(\bW, \bS)$ a \emph{regular stratification of the map germ $G$.}
	
	We say that $G$ is \emph{S-nice} whenever all the above subsets of the target are well-defined as subanalytic germs, 
	independent of the radius $\e$.
	
\end{definition}	

\begin{remark}\label{r:tame}
By Theorem \ref{p:fbarg-tame}, the map germs of type $f\bar g$ such that $f/g$ is an irreducible fraction are S-nice.
\end{remark}

\begin{definition}\label{d:tube1}
Let $G:(\bR^{m},0) \rightarrow (\bR^{p}, 0)$ be a non-constant S-nice analytic map germ.
We say that $G$ has  a \emph{singular Milnor tube fibration} relative to some regular stratification $(\bW, \bS)$, which is well-defined as germ at the origin by our assumption,  if for any small enough $\e > 0$ there exists  $0<\eta \ll \e$ such that the restriction:
\begin{equation}\label{eq:tube1}
G_| :  B^{m}_{\e} \cap G^{-1}( B^{p}_\eta \m \{ 0\} ) \to  B^{p}_\eta \m \{ 0\} 
\end{equation}
is a stratified locally trivial fibration which is independent, up to stratified homeomorphisms, of the choices of $\e$ and $\eta$. 
\end{definition}

What means more precisely ``independent, up to stratified homeomorphisms, of the choices of $\e$ and $\eta$'':  when replacing $\e$ by some $\e'<\e$ and $\eta$ by some small enough $\eta'<\eta$, then the  fibration  \eqref{eq:tube1} and the analogous fibration for $\e'$ and $\epsilon'$
have the same stratified image in the smaller disk $B^{p}_{\eta'} \m \{ 0\}$, and the fibrations are stratified diffeomorphic over this disk. This property is based on the fact that the image of $G$ is well-defined as a stratified set germ, which amounts to our assumption of ``S-nice''.

By \emph{stratified locally trivial fibration} we mean that for any stratum $S_{\beta}$, the restriction $G_{| G^{-1}(S_{\beta})}$ is a locally trivial stratwise fibration.
The fibres are of course empty over the complement of $\im G$, whereas the fibre over some connected stratum $S_{\beta} \subset \im G$ of $\bS$ is a singular stratified set, namely the union of all fibres of the strata of $\bW$ which map onto $S_{\beta}$.
\medskip

We define the \emph{stratwise Milnor set} $M(G)$ as the union of the Milnor sets of the restrictions of $G$ to each stratum. Namely, let $W_\alpha \in \bW$ 
be the germ at the origin of some stratum, and let 
$M(G_{|W_\alpha})$ be the Milnor set as in Definition \ref{d:M}, namely:
\[M(G_{|W_\alpha}):=\left\lbrace x \in W_\alpha \mid \rho_{|W_\alpha} \not\pitchfork_x G_{|W_\alpha} \right\rbrace  \]
where $\rho_{|W_\alpha}$ denotes the restriction of the distance function $\rho$ to the subset $W_\alpha$.
\begin{definition}\label{d:Mstr}
We call $M(G):=\sqcup_{\alpha} M(G_{|W_\alpha})$
	the set of \textit{stratwise $\rho$-nonregular points} of $G$ with respect to the stratifications $\bW$ and $\bS$.
\end{definition}
Note  that $M(G)$  is  closed, due to the fact that $\bW$ is a Whitney (a) stratification. 
We then consider the following condition:
\begin{equation}\label{eq:main2}
\overline{M(G) \m  V_{G}} \cap V_{G} \subset     \{ 0\}.
\end{equation}
and remark that  \eqref{eq:main2} restricted to $M(G)\m G^{-1}(\Disc G)$ is condition \eqref{eq:main}.

In this new setting, the basic existence result is the following:

\begin{theorem} \label{t:tube}
	Let $G:(\bR^m, 0) \to (\bR^p,0)$ be a non-constant S-nice analytic map germ. If $G$ satisfies condition \eqref{eq:main2}, then $G$ has a singular Milnor tube fibration \eqref{eq:tube1}.
\end{theorem}
\begin{proof} 
 Let us  fix a regular stratification $(\bW, \bS)$, cf Definition \ref{d:stratif}.
 Condition \eqref{eq:main2} implies the existence of  $\e_0>0$ such that, for any  $0<\e <\e_0$, there exists $\eta$, $0<\eta \ll \e$, such that every restriction map
\begin{equation}\label{eq:restrG}
G_{|}:  W_{\alpha}\cap \overline{B^m_\e} \cap G^{-1}(B^{p}_\eta \m \{0\}) \to  S_{\beta} \cap  B^{p}_{\eta} \m \{0\}
\end{equation}
is a submersion on a manifold with boundary.  Indeed, 
since the sphere $S^{m-1}_\e$ is transversal to all the finitely many strata of the Whitney stratification $\bW$ at $0\in \bR^{n}$,
it follows that  the intersection $S^{m-1}_\e \cap \bW$ is a Whitney stratification $\bW_{S, \e}$ which refines $\bW$. By condition \eqref{eq:main2}, the map $G$ is not only transversal to  the  stratification $\bW$ but also to the stratification $\bW_{S, \e}$, for any $0<\e < \e_{0}$. It then follows that the map \eqref{eq:restrG} is a stratified submersion and it is proper, thus it is  a stratified fibration by Thom-Mather Isotopy Theorem.  Moreover, condition \eqref{eq:main2} tells that this fibration is independent of $\e$ and $\eta$ up to stratified homeomorphisms. 	
\end{proof}

By construction $\Disc G$ and the complement $\overline{\im G} \m \Disc G$  are unions of strata. In particular,  if 
the singular tube fibration exists, then the Milnor-Hamm fibration exists too.  
Thus, from the above proof one immediately derives:
\begin{corollary}\label{c:tube-hamm}
Under the hypotheses of Theorem \ref{t:tube},  the map $G$ has a Milnor-Hamm fibration over $B^{p}_{\eta} \m \Disc G$, with nonsingular Milnor fibre over each connected component. 
\fin
\end{corollary}

\begin{example}\label{ex:notfull}
Let $G:(\mathbb{R}^3,0)\to (\mathbb{R}^2,0)$,  $G(x,y,z)=(xy,z^2)$. One has: 	
\[ 
\begin{array}{ll}
	V_G =\{x=z=0\}\cup \{y=z=0\}  & \im G = \bR \times \bR_{\ge 0} \subsetneq \bR^{2}  \medskip\\
	\Sing G= \{x=y=0\}\cup \{z=0\}  & G(\Sing G) = \{0\} \times \bR_{\ge 0} \cup  \bR \times \{0\}  \medskip\\
	\Disc G= \{(0,\beta)\,|\,\beta \ge0\}\cup \{(\lambda,0)\,|\, \lambda \in \mathbb{R}\} \medskip\ &
	G^{-1}(\Disc G)=\{x=0\}\cup\{y=0\}\cup \{z=0\}  \medskip\\
	M(G)= \{x=\pm y\}\cup\{z=0\} & \overline{M(G) \m G^{-1}(\Disc G)} = \{x=\pm y\}.\\
	\end{array}	
\]
	\smallskip
It follows that $G$ is nice and  satisfies  the condition \eqref{eq:main}, thus $G$ has a Milnor-Hamm fibration by Lemma \ref{full}.
The complement  $\bR^{2}\m \Disc G$ consists of 3 connected components.  We have: the fibre over $\bR \times \bR_{<0}$ is empty; the fibre over $\bR_{>0} \times \bR_{>0}$ and the fibre over $\bR_{<0} \times \bR_{>0}$  are  two non-intersecting hyperbolas, with 4 connected components.

It moreover follows that $G$ is S-nice and satisfies the condition \eqref{eq:main2},  thus it has a singular tube fibration by Theorem \ref{t:tube}. The discriminant has 4 strata, with singular fibres over each of them, as follows: the positive vertical axis, the fibre over which has two disconnected components each of which being two intersecting lines; the positive and the negative horizontal axis, the fibres over which are both hyperbolas with two components; and the origin, the fibre over which is two intersecting lines.

\end{example}
\subsection{Relation with the Thom regularity}
Let us explain shortly the relation to Thom regularity, extending our discussion in \S \ref{s:partial-thom}.

\begin{definition}\label{d:thom-general} 
Let $G:(\bR^m, 0) \to (\bR^p,0)$ be a non-constant analytic map germ. We say that $G$ is \textit{Thom regular at $V_G$} if there
exists a Whitney stratification $(\bW, \bS)$ like in Definition \ref{d:stratif} such that  $0$ is a point stratum in $\bS$, that $V_{G}$ is a union of strata of $\bW$, and that the Thom (a$_{G}$)-regularity condition is satisfied at any stratum of $V_{G}$.
\end{definition}

Note that we do not ask the full Thom regularity of the map $G$, i.e. we do not ask that $\bW$ is a Thom regular stratification; we only ask that Thom regularity holds at the fibre over 0. This condition is enough to insure condition \eqref{eq:main2} and thus we derive the following statement from Theorem \ref{t:tube} and from the proof of Theorem \ref{c4} adapted to our new setting:

\begin{corollary}\label{c:tube-gen}
	Let $G:(\bR^m, 0) \to (\bR^p,0)$ be a non-constant  S-nice analytic map germ. If $G$ is Thom regular at $V_G$, $\dim V_{G}>0$,  then $G$ has a singular Milnor tube fibration \eqref{eq:tube1}. 

In particular, if  $V_G  \cap \Sing G  = \{0\}$ and $\dim V_{G}>0$, then $G$ has a Milnor-Hamm fibration \eqref{eq:tube}.
	\fin
\end{corollary} 

Theorem \ref{c3}  already provides classes of maps $f\bar{g}$  which are S-nice,  Thom regular at $V_{f\bar{g}}$,
and thus have singular Milnor tube fibration.  Here is an example with this property:

\begin{example} Let $f,g :\mathbb{C}^2 \to \mathbb{C}$ given by $f(x,y) =xy+x^{2}$ and $g(x,y)= y^{2}$. One has $V_{(f,g)} = \{(0,0)\}$ and $\Sing(f,g) = \{y=0\}\cup \{y=-2x\}$, thus $(f,g)$ is obviously  Thom regular. However $\Disc (f,g) = \{(x^2,0)\,|\, x\in\mathbb{C} \}\cup \{(-x^2,4x^2)\,|\, x\in\mathbb{C} \}$ and therefore $f\bar g$ has  non-isolated critical value. It then follows from Corollary \ref{c:tube-gen}  that $f\bar{g}$ is Thom regular, hence it  has a Milnor-Hamm fibration, and also a singular Milnor tube fibration.
\end{example}

\subsection{An example without  Thom regularity}
The singular Milnor tube fibration  may exist without the Thom regularity, as shown by the following example:

\begin{example} 
Let us consider again the real map germ $G$ from Example \ref{ex:1couple}, where it was shown that $G$ is nice but not $\partial$-Thom regular at $V_{G}$, hence not Thom regular at $V_{G}$. Let $(\bW, \bS)$ be the following stratification  of the source and of the target of $G$: $W_{1}:=B_{\e}^{4}\m G^{-1}(\Disc G)$,  $W_{2}:=\{(0,0,z, w)\in \bR^{4} \mid w>0\}$, $W_{3}:=\{(0,0,z, w)\in \bR^{4} \mid w<0\},$ $W_{4}:=\{(0,0,z, 0)\in \bR^{4} \mid z>0\},$ $W_{5}:=\{(0,0,z, 0)\in \bR^{4}\mid z <0\}$, $W_{6}:=\{(0, 0, 0, 0)\}$,  and  $S_{1}:=B_{\eta}^{3}\m \Disc(G)$, $S_{2}:=\{((0,0,\gamma)\in \bR^{3} \mid\gamma >0)\},$ $S_{3}:=\{((0,0,\gamma)\in \bR^{3} \mid \gamma <0)\},$ $S_{4}=S_{5}=S_{6}:=\{(0,0,0)\}$. 

The restriction maps are $G_{j}:= G_{|_{W_{j}}}: W_{j}\to S_{j},$ $j=1,\ldots, 6.$  We had shown in Example \ref{ex:1couple} that $G_{1}$ satisfies condition \eqref{eq:main}. Each restriction map $G_{j}$, $2\leq j\leq 6$, is onto, in particular, it follows that $G$ is also S-nice.

The Milnor sets of the remaining maps are $M(G_{2})=M(G_{3})=$ the $Ow$-axis, and $M(G_{i})\m V_{G}=\emptyset,$ for $i=4,5,6$. It follows that condition \eqref{eq:main2} holds, and thus $G$ has a singular Milnor tube fibration by Theorem \ref{t:tube}. 
 
Since $\Disc G = \{(0,0,\gamma) \in \bR^{3} \mid \gamma\in \bR \}$ does not disconnect the target space, the Milnor fibre outside the discriminant set (i.e. the Milnor-Hamm fibre) is unique and consists of 3 disconnected lines. The fibre over each stratum of the discriminant is a single line. 
\end{example}



\end{document}